\documentclass[12pt]{amsart}
\usepackage{amsfonts}
\usepackage{amsmath}
\usepackage{eurosym}
\usepackage{ifthen}
\usepackage{amssymb}
\usepackage{color}
\usepackage{epsf,graphicx}

\nonstopmode
\numberwithin{equation}{section}

\newtheorem{theorem}{Theorem}[section]
\newtheorem*{ThmA}{Theorem A}
\newtheorem*{ThmB}{Theorem B}
\newtheorem*{ThmC}{Theorem C}
\newtheorem*{ThmD}{Theorem D}
\newtheorem*{ThmE}{Theorem E}

\newtheorem{corollary}[theorem]{Corollary}
\newtheorem{definition}[theorem]{Definition}
\newtheorem{example}[theorem]{Example}

\newtheorem{lemma}[theorem]{Lemma}

\newtheorem{proposition}[theorem]{Proposition}

\theoremstyle{definition}
\newtheorem{remark}[theorem]{Remark}
{\qed\bigskip}

\newcounter{alphabet}
\newcounter{tmp}

\newcounter{minutes}
\divide\time by 60
\newcounter{hours}
\multiply\time by 60
\addtocounter{minutes}{-\time}

\newcommand{\C}{{\mathbb C}}
\newcommand{\D}{{\Delta}}

\renewcommand{\Re}{{\operatorname{Re}\,}}

\newcommand{\hol}{{\operatorname{Hol}}}
\newcommand{\J}{{\mathcal J}}
\newcommand{\N}{{\mathcal N}}
\newcommand{\es}{{\mathcal S}}
\newcommand{\G}{{\mathcal G}}
\renewcommand{\ae}{{\text{a.e.~}}}

\begin{document}

\title[Nonlinear resolvent of holomorphic generators]%
{Geometric properties of the nonlinear resolvent of holomorphic generators}


\author[M. Elin]{Mark Elin}
\address{ORT Braude College,
P.O. Box 78, Karmiel 21982, Israel}
\email{mark$\_$elin@braude.ac.il}
\author[D. Shoikhet]{David Shoikhet}
\address{Holon Institute of Technology}
\email{davidsho@hit.ac.il}
\author[T. Sugawa]{Toshiyuki Sugawa}
\address{Graduate School of Information Sciences,
Tohoku University, Aoba-ku, Sendai 980-8579, Japan}
\email{sugawa@math.is.tohoku.ac.jp}

\begin{abstract}
Let $f$ \ be the infinitesimal generator of a one-parameter semigroup $%
\left\{ F_{t}\right\} _{t\ge0}$ of holomorphic self-mappings of the open unit
disk $\D$.
In this paper we study properties of the family $R$ of resolvents
$(I+rf)^{-1}:\D\to\D~ (r\ge0)$
in the spirit of geometric function theory.
We discovered, in particular, that $R$ forms an inverse L\"owner chain of
hyperbolically convex functions.
Moreover, each element of $R$ satisfies
the Noshiro-Warschawskii condition  and is a starlike function of order at least $\frac12$\,.
This, in turn, implies that each element of
$R$ is also a holomorphic generator.
We mention also quasiconformal extension of an element of $R.$
Finally we study the existence of repelling fixed points of this family.
\end{abstract}
\thanks{
T.~S.~was supported in part by JSPS KAKENHI Grant Number JP17H02847.
}
\maketitle

\section{Introduction}

We denote by $\hol(D_1,D_2)$ the set of holomorphic mappings of a domain $D_1$ into another $D_2.$
If $D$ is a domain in $\C$, then the set $\hol(D) := \hol(D,D)$ forms a semigroup with
composition being the semigroup operation.

\begin{definition}
\label{def G1} A family $\{F_{t}\}_{t\geq 0}$ of functions in $\hol(D )$
is called a one-parameter continuous semigroup on $D$ if the following conditions hold:
\begin{enumerate}
\item[1)] $F_t(z)$ converges to $z$ uniformly on each compact subset of $D$
as $t\to0^+$ and
\item[2)] $F_{t}(F_{s}(z))=F_{t+s}(z)$, whenever $t,s\geq 0$ and $z\in D $.
\end{enumerate}
\end{definition}

It is well known as the Berkson-Porta Theorem  \cite{B-P} that the limit%
\begin{equation*}
\lim_{t\rightarrow 0^{+}}\frac{1}{t}(z-F_{t}(z))=:f(z)
\end{equation*}%
exists with $f\in \hol(D,\C)$ in the topology of locally uniform convergence on $D$ and that $F_t(z)$ is reproduced by $u(t)=F_t(z),$ where $u(t)$ is the solution to the initial value problem of the ODE
\begin{equation*}
\left\{
\begin{array}{c}
\dfrac{du}{dt}+f(u(t))=0, \quad t\ge0,\\
\\
u(0)=z,%
\end{array}%
\right.
\end{equation*}%
(see also \cite{SD} and \cite{E-S-book}).
This function $f$ is called the \textit{(infinitesimal) generator }of the semigroup $\{F_{t}\}_{t\geq 0}$. The set of generators $f$ on $D$ arising in this way will be denoted by $\mathcal{G}(D) $.
We note, in particular, that for some $z_0\in D,$ the equality $F_t(z_0)=z_0$  holds for all $t\ge0$
if and only if $f(z_0)=0.$

The following fact was proved in \cite{RS-SD-98B} (see also \cite{R-S1}).

\begin{ThmA}
\label{teor GA} Let $D$ be a bounded convex domain in $\C$. A function $f\in \hol(D,\C)$ belongs to the class
$\mathcal{G}(D) $ if and only if for every $r\geq 0$ and every $w\in D $ the equation%
\begin{equation}
z+rf(z)=w  \label{G*}
\end{equation}%
has a unique solution $z=\mathcal{J}_{r}(w)\left( =\left( I+rf\right)
^{-1}\left( w\right) \right) $ in $D.$
Moreover, $\J_r(w)$ is holomorphic in $w\in D $.
\end{ThmA}

\bigskip

This solution is called the \textit{nonlinear resolvent }of $f$.
A proof of the existence of $\J_r$ will be given in Section \ref{sec:Proof}
under a stronger assumption.
Various properties of the nonlinear resolvent, like resolvent identities, asymptotic
behaviour, {\it e.t.c.} can be found in the books \cite{R-S1} (for Banach spaces) and
\cite{SD} (for the one dimensional case). In particular, the following
exponential formula holds.

\begin{ThmB}
\label{teorGB} Let $f$ be the generator of a one-parameter semigroup $\{F_{t}\}_{t\geq 0}$ of holomorphic self-mappings of  a bounded convex domain  $D $ and let
\begin{equation*}
\mathcal{J}_{r}=(I+rf)^{-1}\quad (r\ge0)
\end{equation*}%
be the resolvent family of $f$. Then for each $t\geq 0$%
\begin{equation*}
\lim_{n\rightarrow \infty }\mathcal{J}_{\frac{t}{n}}^{n}=F_{t}\text{.}
\end{equation*}
\end{ThmB}

Here we denote by $G^{n}$ the $n$-th iterate of a self-mapping $G$ of $D;$ namely,
$G^{1}=G$, $G^{n}=G\circ G^{n-1}$, $n=1,2,\dots,$ and the limit here and hereafter,
unless otherwise stated, will be in the topology of locally uniform convergence on $D.$

In this paper we mostly deal with the case where  $D$ is the open unit disk
$\D=\{z: |z|<1\}$ in the complex plane $\C.$
Various representations of the class $\mathcal{G}(\D )$ can be found in the
books \cite{SD}, \cite{R-S1} and \cite{E-S-book}. For our purposes we need the following.

\begin{ThmC}
\label{teorA} Let $f\in \hol(\D , \C)$. Then $f\in \mathcal{G}(\D )$ if and only if one of the following conditions holds:
\begin{enumerate}
\item[(i)] $($Berkson-Porta representation \cite{B-P}$)$\quad
there exists a point $\tau$ in the closed unit disk $\overline\D$ such that
\begin{equation*}
f(z)=(z-\tau )(1-z\overline{\tau })p(z),\quad z\in\D;
\end{equation*}
\item[(ii)] $($Aharonov-Elin-Reich-Shoikhet criterion \cite{A-E-R-S}$)$ \quad
there exists a function $q\in\hol(\D,\C)$ with $\Re q\ge0$ such that
\begin{equation*}
f(z)=zq(z)+f(0)-\overline{f(0)}z^{2},\quad z\in\D.
\end{equation*}
\end{enumerate}
\end{ThmC}

The equivalence of conditions (i) and (ii) was shown  in \cite{A-E-R-S}
by using direct complex analytic methods.
Note also that condition (ii) can be written as
\begin{equation*}
\Re (f(z)\overline{z})\geq (1-|z|^{2})\,\Re (f(0)\overline{z})\quad z\in \D .
\end{equation*}

Observe that the first term $f_{1}(z)=z\cdot q(z)$ of the
decomposition formula in (ii) is also an element of $\mathcal{G}(\D )$
with $f_{1}(0)=0$, while the remainder term $f_{2}(z)=a-\overline{a}%
z^{2}$ with $a=f(0)$ is the generator of a one-parameter group of hyperbolic
automorphisms of $\D $.
This implies, in turn, that the set $\mathcal{G}(\D )$ is a real cone in
$\C$ and if $f=h+g$ for some $h,~g\in\mathcal{G}(\D )$
generating the semigroups $\{H_{t}\}_{t\geq 0}$ and $\{G_{t}\}_{t\geq 0}$,
respectively, then $f\in \mathcal{G}(\D )$ and the semigroup $\{F_{t}\}$
generated by $f$ can be reproduced by the so-called product formula%
\begin{equation*}
F_{t}=\lim_{n\rightarrow \infty }\left[ H_{\frac{t}{n}}^{n}\circ G_{\frac{t}{%
n}}^{n}\right] \text{.}
\end{equation*}

Since presentation (i) in Theorem C is
unique, it follows that $f\in \mathcal{G}(\D )$ must have at most one null point (zero) in $\D $. This point $\tau $ is known to be the\textit{\
Denjoy-Wolff point} for the semigroup  $\mathcal{F}=\left\{ F_{t}\right\} _{t\geq 0}$ generated by $f$, that is, if $\mathcal{F}$
contains neither an elliptic automorphism of $\D $ nor the identity mapping, then
\begin{equation}
\underset{t\rightarrow \infty }{\lim }F_{t}\left( z\right) =\tau \quad
\text{  in }~ z\in \D .  \label{D-Wp'}
\end{equation}
The constant mapping $\tau $ is also the limit point for the resolvent family.
Namely,
\begin{equation}
\lim_{r\to\infty}\mathcal{J}_{r}( w) =%
\lim_{r\to\infty}(I+rf)^{-1}(w) =\tau, \quad
w\in \D .  \label{implicit}
\end{equation}
Moreover, for each fixed $r>0$,
\begin{equation}
\lim_{n\to\infty}\mathcal{J}_{r}^{n}\left( w\right)
=\lim_{n\to\infty}(I+rf)^{-n}\left( w\right) =\tau, \quad
w\in \D .  \label{imp2}
\end{equation}%
These assertions may be considered as explicit and implicit continuous analogs
of the classical Denjoy-Wolff Theorem (see, for example, \cite{R-S1} and
\cite{SD}). Note that, in contrast to the formula (\ref{D-Wp'}), the formulae (\ref{implicit}) and (\ref{imp2}) are valid for all $f\in \mathcal{G}(\D ).$

Regarding the boundary behaviour, it is known that if $f\in \mathcal{G}(\D )$ has a boundary regular null point, then this point is a boundary regular fixed point of each element $F_{t}$ of the semigroup generated by $ f$  (see, for example, \cite{SD2, C-D-P}  and \cite{E-S-book}). However  this fact is no longer true for all elements of the resolvent family $\left\{ \mathcal{J}_{r}\right\} _{r\ge0}.$ We study this situation in more detail in Section 5.

\section{Decay properties of the semigroup}

Recall that the semigroup $\{F_t(z)\}_{t\ge0}$ can be reproduced by its generator $f$ as the unique solution of
the differential equation
\begin{equation}\label{eq:DE}
\frac{\partial F_t(z)}{\partial t}+f(F_t(z))=0,\quad t\ge0,
\end{equation}
which satisfies the initial condition $F_0(z)=z.$
In what follows, we restrict ourselves on the case where the generator $f$ satisfies $f(0)=0$,
so that $F_t(0)=0$ for $t\ge0$ and the origin is the
Denjoy-Wolff point of the semigroup.
By Theorem C (ii), we see that a function $f\in\hol(\D,\C)$ with $f(0)=0$
belongs to the class $\mathcal{G}(\D)$ if and only if $\Re[f(z)/z]\ge 0,~ z\in\D.$
Here the value of $f(z)/z$ at $z=0$ is understood to be $f'(0).$
In the sequel, we will assume that $\Re[f(z)/z] \neq 0$;
otherwise $f(z)/z$ is identically a purely imaginary constant; that is,
$f(z)=aiz$ for a real constant $a$, hence the semigroup consists of just rotations
or the identity mapping; namely, $F_t(z)=ze^{-ati},$  so that
$\{F_t\}_{t\ge0}$ does not have the Denjoy-Wolff point at $z=0.$
We denote by $\N$ the set of such functions $f$; that is,
\begin{equation*}
\mathcal{N}=\left\{ f\in \mathop{\rm Hol}\nolimits(\D ,\mathbb{C}):\
f(0)=0,\ \operatorname{Re}\frac{f\left( z\right) }{z}> 0\right\} .
\end{equation*}%
We study geometric properties of the resolvent family of a function
in this class, which may be of independent interest.

The class $\mathcal{N}$ has been studied independently in the framework of
geometric function theory (see, for example, \cite{marx}, \cite{stro}, \cite%
{suf}, \cite{G-K}\textbf{\ }and references therein) with its relations to
the classes of convex and starlike functions. In particular, classical
results in \cite{marx}, \cite{stro} and Theorem C above imply that
\textit{each convex function }$h\in \mathop{\rm Hol}\nolimits(\D ,%
\mathbb{C})$\textit{\ with }$h(0)=0$\textit{\ is an element of }$\mathcal{N}%
. $

Also, it is a simple exercise to show that the class $\mathcal{NW}$ of
functions $f\in\hol(\D, \C)$ with $f(0)=0$
satisfying the Noshiro-Warschawski condition
\begin{equation}
\operatorname{Re}f'(z) >0,\quad z\in \D,  \label{N_W}
\end{equation}
consists of univalent functions and is contained in $\mathcal{N}\subset \mathcal{G%
}\left( \D \right) $. See for details and more results \cite{Goodman},
\cite{duren}, \cite{FSG}, \cite{E-S-S} and \cite{SD1}.

Here we consider properties of semigroups and resolvent families of
elements in $\mathcal{N}$.
For the semigroup $\{F_{t}\}_{t\geq 0}$ generated by a function $f\in \mathcal{N},$
a sharp estimate for the rate of convergence to the origin
is established by Gurganus \cite{GKR-75}
(see also \cite{SD} and \cite{E-S-book} for details):
\begin{equation*}
|F_{t}\left( z\right) |\leq |z|\exp\left(-t \, \Re f'(0)\frac{1-|z|}{1+|z|}\right).
\end{equation*}
However, it should be mentioned that this estimate is \emph{not} uniform on
$\D $.
The following result (see \cite{FSG}) gives us a criterion for uniform decay of $|F_t(z)|.$
Since in \cite{FSG} the additional condition $f'(0)=1$ was assumed, we reproduce its proof
without unnecessary restrictions for convenience of the reader.

\begin{lemma}
\label{lem:kappa}
Let $\kappa>0$ be a constant.
Then the semigroup $\{F_{t}\}_{t\geq 0}$ generated by an
$f\in \N$ has the uniform exponential rate of convergence
\begin{equation}\label{eq:kappa1}
|F_{t}(z)|\leq |z|e^{-\kappa t },\quad t\in[0,\infty), z\in\D,
\end{equation}%
if and only if
\begin{equation}\label{eq:kappa2}
\operatorname{Re}\frac{f(z)}{z}\geq \kappa,\quad z\in\D.
\end{equation}
\end{lemma}

\begin{proof}
For a fixed $z\in\D,$ we put
$$
g(t)=|F_t(z)|^2-|z|^2e^{-2\kappa t}=F_t(z)\overline{F_t(z)}-|z|^2e^{-2\kappa t},
$$
so that
\begin{align}\label{aux1}
g'(t)
&=-2\Re\left[\overline{F_t(z)}f(F_t(z))\right]+2\kappa|z|^2e^{-2\kappa t}\nonumber \\
&=-2\Re\left[\overline{F_t(z)}f(F_t(z))\right]+2\kappa\left(|F_t(z)|^2-g(t)\right),
\end{align}
where we used \eqref{eq:DE}.
We assume condition \eqref{eq:kappa1} so that $g(t)\le 0$ for $t\ge0.$
Then, because of $g(0)=0,$ we get
$$
0\ge \lim_{t\to0^+}\frac{g(t)}{t}=g'(0)=-2\Re\left[\overline z f(z)\right]+2\kappa|z|^2,
$$
from which \eqref{eq:kappa2} follows.

Next we assume condition \eqref{eq:kappa2} and put $h(t)=g'(t)+2\kappa g(t).$
Then by \eqref{aux1}, $h(t)\le 0$ for $t\ge0.$
Since $[e^{2\kappa t}g(t)]'=e^{2\kappa t}h(t),$ one has
$$
e^{2\kappa t}g(t)=\int_0^t e^{2\kappa \tau}h(\tau)d\tau\le 0,
$$
which implies \eqref{eq:kappa1}.
\end{proof}

The number $\kappa $ satisfying \eqref{eq:kappa1} is called an \textit{exponential
squeezing coefficient.} For instance, if $\{F_{t}\}_{t\geq 0}$\ is
generated by $f\in $ $\mathcal{NW}$ with $f^{\prime }\left( 0\right) =1$,
then it converges uniformly to the origin and has the exponential squeezing
coefficient\textit{\ }%
\begin{equation*}
\kappa =2 \log 2-1=0.386\cdots.
\end{equation*}%
This estimate is sharp.
See \cite{FSG}, \cite{E-S-S} and \cite{SD1} for the proof of this fact.

\begin{definition}
\label{star}A function $f\in \hol(\D, \C)$ with $f\left( 0\right) =0,~ f'(0)\ne 0$
is said to be starlike of order $%
\alpha \in \lbrack 0,1)$ if
\begin{equation*}
\operatorname{Re}\frac{zf^{\prime }\left( z\right) }{f\left( z\right) }>\alpha ,%
\quad z\in \D .
\end{equation*}
The set of starlike functions $f$ of order $\alpha \in \lbrack 0,1)$ with $f'(0)>0$
will be denoted by $\es^{\ast }\left( \alpha \right) .$
\end{definition}

Due to the
Nevanlinna-Alexander criterion (see, for example, \cite{Goodman}), $%
\es^{\ast }\left( 0\right) $ is the set of those univalent functions $f$ on
$\D $ with $f(0)=0,~ f'(0)>0$ for which $f(\D)$ is a starlike domain
with respect to the origin.

The following assertion is an extension of classical results of Marx \cite{marx} and
Strohh\"{a}cker \cite{stro} (see also \cite[Theorem 2.6a]{M-M}).

\begin{lemma}
\label{lem:ss}
For $\alpha \in \lbrack 0,1),$ the implication relation $\mathcal{S}%
^{\ast }(\alpha )\subset \mathcal{N}$ holds if and only if $\alpha \geq \frac{1}{2}$.
If a function $f$ with $f'(0)=\beta>0$ belongs to the class $\mathcal{S}^{\ast }(\alpha )$
for some $\alpha \geq \frac{1}{2}, $
then $2^{-2(1-\alpha)}\beta$ is an exponential squeezing coefficient for $f.$
That is to say, the semigroup $\{F_{t}\left( z\right) \}_{t\geq 0}$ generated
by $f$ has the following uniform rate of convergence:
\begin{equation*}
|F_{t}\left( z\right) |\leq |z| \exp\left[- 2^{-2(1-\alpha )}\beta t\right],\quad z\in\D.
\end{equation*}
Moreover, the following inequality holds and the bound $(1-\alpha)\pi$ is sharp:
\begin{equation}\label{eq:sector}
\left|\arg\frac{f(z)}{z}\right|< (1-\alpha)\pi,\quad z\in\D.
\end{equation}
\end{lemma}

Before the proof, we recall the notion of subordination.
A function $f\in\hol(\D,\C)$ is said to be subordinate to another
$g\in\hol(\D,\C)$ and written as $f\prec g$ or $f(z)\prec g(z)$
if there is a function $\omega\in\hol(\D)$
with $\omega(0)=0$ such that $f=g\circ\omega.$
When $g$ is univalent, this means that $f(0)=g(0)$ and $f(\D)\subset
g(\D).$

\begin{proof}[Proof of the lemma]
By a theorem of Pinchuk \cite[Theorem 10]{Pin68}, we have
$f(z)/zf'(0)\prec f_\alpha(z)/z,$ where $f_{\alpha }(z)=z(1-z)^{-2(1-\alpha )}.$
Note that $f_\alpha\in \mathcal{S}^{\ast }(\alpha )$.
Since
$$
\arg\frac{f_\alpha(z)}{z}=-2(1-\alpha)\arg(1-z),\quad z\in\D,
$$
the supremum of $|\arg[f_\alpha(z)/z]|$ over $z\in\D$ is exactly $(1-\alpha)\pi,$
which implies \eqref{eq:sector}.
In particular, $\Re[f_\alpha(z)/z]>0$ holds precisely when $\frac12\le \alpha.$
Thus the necessity part has been proved.

Next we assume that $\frac12\le \alpha<1.$
Again by Pinchuk's theorem, we have
$$
\inf_{z\in\D}\Re\frac{f(z)}{\beta z}
\ge\inf_{z\in\D}\Re\frac{f_\alpha(z)}{z}=2^{-2(1-\alpha)},
$$
where we have used the fact that $f_\alpha(z)/z=(1-z)^{-2(1-\alpha)}$
is convex univalent on $\D$ for $\alpha\ge\frac12.$
The remaining part of the lemma follows from Lemma \ref{lem:kappa}.
\end{proof}

We now recall the notion of hyperbolically convex functions
that were studied by many authors and characterized in different aspects;
see for instance \cite{MM99}, \cite{MP00}, \cite{BPW} and references therein.

\begin{definition}
A univalent function $f\in\hol(\D)$ is called  hyperbolically convex if its image $D=f(\D)$
is a hyperbolically convex domain in the sense that for every pair of points $w_1, w_2\in D$,
the hyperbolic geodesic segment joining $w_1$ and $w_w$ in $\D$ lies entirely in $D.$
\end{definition}

Among other important properties of such functions, we recall the result due to
Mej\'\i a and Pommerenke \cite{MP00}  that
a hyperbolically convex function $f$ with $f(0)=0$ and $ f'(0)\ne0$ is starlike
of order $\frac12.$

We are now in a position to formulate our main results.

\begin{theorem}
\label{thm:NW} Let $\{%
\mathcal{J}_{r}\}_{r\geq 0}$ be the resolvent family of a function $f\in\N:$
\begin{equation*}
\mathcal{J}_{r}=(I+rf)^{-1}\in\hol(\D ), \quad r\geq 0.
\end{equation*}
Then for each $r\geq 0,$ the resolvent $\mathcal{J}_{r}$ belongs to $\mathcal{NW},$ that is,
\begin{equation}
\operatorname{Re}\mathcal{J}_{r}'(w)>0, \quad w\in\D,  \label{N-W}
\end{equation}%
and  $\J_r$ is hyperbolically convex.
Moreover,
\begin{equation}
\mathcal{J}_{r}(\D )\subset \mathcal{J}_{s}(\D ),  \quad 0\le s<r.
\label{seq}
\end{equation}%
\end{theorem}

Note that $\J_r(0)=0$ and $\J_r'(0)=1/(1+rf'(0))$ (see \eqref{eq:J'(0)} below).
Since a hyperbolically convex function is starlike of order $\frac12$ as is mentioned above,
we obtain the following corollary.
Here, we also note that the Marx-Strohh\"acker theorem \cite{marx}, \cite{stro} states also that
a function $f\in\mathcal{S}^*(\frac12)$ satisfies $\Re[f(z)/zf'(0)]>1/2$ in $|z|<1$
(see also \cite[Theorem 2.6a]{M-M}).
Hence we obtain the following corollary.

\begin{corollary}\label{cor2}
Under the same assumptions as in Theorem \ref{thm:NW},
for each $r\ge 0,$ $\J_r$ is starlike of order $\frac12;$ namely,
\begin{equation}
\operatorname{Re}\frac{w\mathcal{J}_{r}^{\prime }(w)}{\mathcal{J}_{r}(w)}
>\frac{1}{2}\,,   \quad w\in\D. \label{G1.1}
\end{equation}%
If, in addition, $f'(0)=\beta>0,$ then
\begin{equation}
\operatorname{Re}\frac{\mathcal{J}_{r}\left( w\right) }{w}>\frac{1}{2(1+\beta r)}\,,
\quad w\in\D. \label{m-s}
\end{equation}%
In particular, for each $r>0,$
the semigroup generated by $\mathcal{J}_{r}$ converges to $0$
uniformly on $\D $ with exponential squeezing coefficient $\kappa =1/[2(1+\beta r)] $.
\end{corollary}

We illustrate Theorem \ref{thm:NW} and Corollary \ref{cor2} by the following example.

\begin{example}
Consider the semigroup generator $f(z)=z/(1-z)$ in $\N.$
Solving the equation
$$
  z+rf(z)=z\left(1+\frac{r}{1-z}\right)=w
$$
in $z,$ we find its resolvent
$$
\mathcal{J}_r(w)=\frac{r+1+w}2-\frac12\sqrt{(r+1+w)^2-w}
$$
and then calculate
$$
\frac{w{\mathcal{J}_r}'(w)}{\mathcal{J}_r(w)}
=\frac12+ \frac{1+r-w}{2\sqrt{(1+r+w)^2-4w}}\,.
$$
For every fixed $r$, the real part of the last expression tends to its minimum as $w\to1$ and hence
  \[
  \inf\Re \frac{w{\mathcal{J}_r}'(w)}{\mathcal{J}_r(w)}=\frac12 + \frac12\sqrt{\frac{r}{r+4}}\,,
    \]
  so $\mathcal{J}_r$ is a starlike function of order $\frac12 + \frac12\sqrt{\frac{r}{r+4}}\,,$ which tends to $\frac12$ as $r\to0^+$.
\end{example}

\section{Proof of the main theorem}
\label{sec:Proof}

For the sake of completeness, we start this section with the following useful sufficient condition
(see \cite{SD} and \cite{R-S1}) for $f\in \hol(\D ,\C)$ to be an infinitesimal generator.

\begin{lemma}
\label{sufficient}
Let $f\in \hol(\D ,\C)$.
Suppose that there exists an $\varepsilon\in(0,1)$ such that
\begin{equation*}
\operatorname{Re}\left[ f(z)\overline z \right] >0 \quad
\text{for}~ z\in\D ~\text{with}~ |z|\ge 1-\varepsilon.
\end{equation*}%
Then
$f\in \mathcal{G}(\D )$.
\end{lemma}

\begin{proof}
By Theorem A, it is enough to show existence of
the nonlinear resolvent of $f$ for all $r>0.$
For fixed $0<r<1$ and $w\in\D,$ put
\begin{equation}\label{eq:g}
g(z)=z+rf(z)-w.
\end{equation}
Choose $t$ so that $\max\{1-\varepsilon, |w|\}<t<1.$
Then, for $|z|=t,$
\begin{align*}
\Re\big[g(z)\overline z\big]
&\ge |z|^2+r\Re\big[ f(z)\overline z\big] -|zw| \\
&>|z|(|z|-|w|)=t(t-|w|)>0.
\end{align*}
Therefore, by the argument principle, we see that
the number of zeros of $g(z)/z$ in $|z|<t$ is the same as
that of poles of $g(z)/z,$ which is $1.$
Thus the function $g(z)$ has a unique zero in the unit disk $\D.$
So, the result follows.
\end{proof}

We also need the following result for the proof of our main theorem.
This assertion was first conjectured by Mej\'{\i}a and Pommerenke \cite{MP05}
and proved by Solynin \cite{Sol07}.

\begin{ThmD}\label{Thm:Solynin}
Let $\varphi\in\hol(\D)$ satisfy $\varphi(0)\ne0.$
Then the open set $\Omega=\{z\in\D: |z|<|\varphi(z)|\}$ is hyperbolically convex
in $\D.$
\end{ThmD}

We are ready to prove our main theorem.

\begin{proof}[Proof of Theorem \ref{thm:NW}]
\label{proofG1}
Fix $r\in(0,\infty).$
Differentiating the resolvent equation \eqref{G*} in $w\in \D,$
we get the relation
\begin{equation}\label{eq:J'}
\J_r'(w)=\frac1{1+rf'\big(\J_{r}(w)\big)}
\end{equation}
and, in particular, the formula
\begin{equation}\label{eq:J'(0)}
\J_r'(0)=\frac{1}{1+rf'(0)}.
\end{equation}
For a fixed $w\in\D,$ we define $g$ by \eqref{eq:g}.
Then by Lemma \ref{sufficient} and its proof, we observe that
$g$ also belongs to $\G(\D).$
Since $\J_r(w)$ is an interior null point of $g,$ by the
Berkson-Porta formula (condition (i) in Theorem C),
the function $g$ can be represented in the form
\begin{equation*}
g(z) =\big( z-\mathcal{J}_{r}(w)\big) \big( 1-\overline{%
\mathcal{J}_{r}(w)}z\big) p(z) ,
\end{equation*}%
where $\operatorname{Re}p\left( z\right) >0$ for $z\in \D .$
In particular,  \eqref{eq:J'} implies
\begin{equation*}
\operatorname{Re}g'\big(\mathcal{J}_{r}(w)\big) =\left( 1-\left\vert
\mathcal{J}_{r}(w)\right\vert ^{2}\right) \operatorname{Re}p\big(\J_r(w)\big)>0.
\end{equation*}%
On the other hand, by \eqref{eq:g}, $g'\big(\J_{r}(w)\big) =1+rf'\big(\J_{r}(w)\big).$
Hence
$$
\Re\J_r'(w)
=\Re\frac1{1+rf'\big(\J_{r}(w)\big)}
=\Re\frac1{g'\big(\J_{r}(w)\big)}>0,
$$
which proves \eqref{N-W}.

Next we show the hyperbolic convexity of $\mathcal{J}_r.$
We note that the function $h(z)=1+rf(z)/z$ satisfies
$$
|h(z)|\ge\Re h(z)=1+r\,\Re\frac{f(z)}{z}>1.
$$
Therefore, $\varphi(z)=1/h(z)$ satisfies the assumptions of Theorem D.
On the other hand, for $z\in\D,$
\begin{align*}
z\in\J_r(\D)
&\Leftrightarrow
z+rf(z)\in\D \\
&\Leftrightarrow
|z+rf(z)|=|z||h(z)|<1 \\
&\Leftrightarrow
|z|<|\varphi(z)|.
\end{align*}
Thus we now conclude that $\J_r(\D)=\{z\in\D: |z|<|\varphi(z)|\}$
is hyperbolically convex by Theorem D.

To prove the inclusion (\ref{seq}), we note
that a point $z$ in $\D $ belongs to $\mathcal{J}_{r}(\D
) $ if and only if $(I+rf)(z)\subset \D$. Thus it is enough to show
that $\left\vert z+sf\left( z\right) \right\vert \le\left\vert z+rf\left(
z\right) \right\vert $ for $z\in\D$ whenever $r>s\geq 0.$
Indeed, since $\operatorname{Re}\big[f(z) \overline{z}\big]\geq 0$
this inequality follows from
\begin{eqnarray*}
\left\vert z+sf\left( z\right) \right\vert ^{2} &=&\left\vert z\right\vert
^{2}+2s\operatorname{Re}f\left( z\right) \overline{z}+s^{2}\left\vert f\left(
z\right) \right\vert ^{2}\\
&\le&\left\vert z\right\vert ^{2}+2r\operatorname{Re}f\left( z\right) \overline{z}%
+r^{2}\left\vert f\left( z\right) \right\vert ^{2}=\left\vert z+rf\left(
z\right) \right\vert ^{2}.
\end{eqnarray*}%
\end{proof}

\begin{remark}
The inequality \eqref{G1.1} was derived via Solynin's theorem and
a result of Mej\'\i a-Pommerenke.
We can, however, show it directly.
Indeed, by using the above notation, we have
\begin{equation}\label{eq:F}
w\varphi\big(\J_r(w)\big)=\frac{w\J_r(w)}{\J_r(w)+rf\big(\J_r(w)\big)}
=\J_r(w).
\end{equation}
This means that the function $F(z)=w\varphi(z)$ fixes
the point $z=\J_r(w).$
In particular, the Schwarz-Pick lemma implies that
$|w\varphi'(\J_r(w))|=|F'(\J_r(w))|<1,$
where we used the fact that $F$ is not a disk automorphism.
Differentiating both sides of \eqref{eq:F} gives us
$$
w\J_r'(w)
=\frac{w\varphi\big(\J_r(w)\big)}{1-w\varphi'\big(\J_r\big)}
=\frac{\J_r(w)}{1-w\varphi'\big(\J_r\big)}\,,
$$
hence
$$
\Re\frac{w\J_r'(w)}{\J_r(w)}=\Re\frac{1}{1-w\varphi'\big(\J_r\big)}>\frac12\,.
$$
\end{remark}

\section{Inverse L\"owner chains}

Theorem \ref{thm:NW} tells us that $\Omega_r=\J_r(\D),~ 0\le r<\infty,$ is a
decreasing family of domains in the unit disk $\D.$
We can thus introduce some aspects of L\"owner theory.
Indeed, we will give another proof of the above fact later.
The authors believe that it leads to more geometric understandings of the
family of nonlinear resolvents for $f\in\N.$

\begin{definition}
A map $p:\D\times[0,+\infty)\to\C$ is called a Herglotz function of divergence type
if the following three conditions are satisfied:
\begin{enumerate}
\item[(a)] $p_t(z)=p(z,t)$ is analytic in $z\in\D$ and measurable in $t\ge0,$
\item[(b)] $\Re p(z,t)>0~(z\in\D,~ \ae t\ge0),$
\item[(c)] $p(0,t)$ is locally integrable in $t\ge0$ and
$$
\displaystyle\int_0^\infty \Re p(0,t) dt=+\infty.
$$
\end{enumerate}
\end{definition}

Note that the term {\it Herglotz function of order $d$}
is used in \cite{BCD} to mean the function $p(z,t)$ with the divergence condition
being replaced by $L^d([0,\infty))$-convergence in the above definition.

The following result was proved by Becker \cite[Satz 1]{Becker76}.

\begin{ThmE}\label{Thm:LPDE}
Let $p(z,t)$ be a Herglotz function of divergence type.
Then there exists a unique solution $f_t(z)=f(z,t),$ which is analytic and univalent
in $|z|<1$ for each $t\in[0,+\infty)$ and locally absolutely continuous in $0\le t<\infty$ for each $z\in\D,$
to the differential equation
\begin{equation}\label{eq:LPDE}
\dot f(z,t)=z f'(z,t)p(z,t) \quad (z\in\D, \ae t\ge0)
\end{equation}
with the normalization conditions $f_0(0)=0$ and $f_0'(0)=1.$
Moreover, the solution satisfies $f_s\prec f_t$ for $0\le s\le t.$
\end{ThmE}

Here and hereafter, we write
$$
\dot f(z,t)=\frac{\partial}{\partial t}f(z,t),\qquad
f'(z,t)=\frac{\partial}{\partial z}f(z,t).
$$
In addition, in the proof of Theorem~E, Becker  showed  the formula
\begin{equation}\label{eq:a1}
f'(0,t)=\exp\left(\int_0^t p(0,t)dt\right),
\end{equation}
so that $|f'(0,t)|=\exp\big(\int_0^t\Re p(0,t)dt\big)\to+\infty$
as $t\to+\infty.$  We remark that the uniqueness assertion is no longer valid if we drop
the univalence condition on $f_t.$
For instance, $\tilde f(z,t)=\Phi(f(z,t))$ satisfies \eqref{eq:LPDE}
as well as $\tilde f(0,0)=0$ and $\tilde f'(0,0)=1$ when
$\Phi$ is an entire function with $\Phi(0)=0$ and $\Phi'(0)=1.$

We now make a definition after Betker \cite{Bet92}.

\begin{definition}
A family of analytic functions $g_t(z)=g(z,t)~(0\le t<\infty)$ on the unit disk $\D$
is called an {\it inverse L\"owner chain} if the following conditions are satisfied:
\begin{enumerate}
\item[(i)]
$g_t:\D\to\C$ is univalent for each $t\ge 0,$
\item[(ii)]
$g_t\prec g_s$ whenever $0\le s\le t,$
\item[(iii)]
$b(t)=g_t'(0)$ is locally absolutely continuous in $t\ge0$ and
$b(t)\to 0$ as $t\to\infty.$
\end{enumerate}
\end{definition}

Note that condition (ii) means that $g_t(\D)\subset g_s(\D)$
and $g_t(0)=g_s(0)$ for $0\le s\le t.$
Condition (iii) implies that $g_t(z)\to 0$  locally uniformly on $|z|<1$ as $t\to+\infty.$
The following lemma gives us sufficient conditions for $g(z,t)$
to be an inverse L\"owner chain.

\begin{lemma}\label{lem:IPDE}
Let $g_t(z)=g(z,t)$  be a family of analytic functions on $\D$ for $0\le t<\infty$
with the following properties:
\begin{enumerate}
\item[1)] $g_t$ is univalent on $\D$ for each $t\ge0,$
\item[2)] $g(0,s)=g(0,t)$ for $0\le s\le t,$
\item[3)] $g(z,0)=z$ for $z\in\D,$
\item[4)] $g(z,t)$ is locally absolutely continuous in $t\ge0$ for each $z\in\D,$
\item[5)] the differential equation
\begin{equation}\label{eq:IPDE}
\dot g(z,t)=-z g'(z,t)p(z,t) \quad (z\in\D, \ae t\ge0).
\end{equation}
holds for a Herglotz function $p(z,t)$ of divergence type.
\end{enumerate}
Then $g_t(\D)\subset g_s(\D)$ for $0\le s\le t.$
Also, $g_t'(0)$ is locally absolutely continuous in $t\ge0$ and
tends to $0$ as $t\to+\infty.$ 
\end{lemma}

\begin{proof}
We follow Betker's method in \cite{Bet92}.
First we note that $g(0,t)=g(0,0)=0$ for $t\ge0$ by conditions 2) and 3).
Fix any $T>0$ and define a new family of functions $f_t(z)=f(z,t)$ by
$$
f(z,t)=\begin{cases}
g(z,T-t) &\quad (0\le t\le T), \\
e^{t-T}z &\quad (T\le t<\infty).
\end{cases}
$$
Then the family $f(z,t)$  satisfies the L\"owner equation
$$
\dot f(z,t)=z f'(z,t)q(z,t) \quad (z\in\D, \ae t\ge0),
$$
where
$$
q(z,t)=\begin{cases}
p(z,T-t) &\quad (0\le t\le T), \\
1 &\quad (T<t<\infty).
\end{cases}
$$
It is easy to check that $q(z,t)$ is a Herglotz function of divergence type.
Now Theorem E implies that
$f(z,t)/f'(0,0)$ is a  L\"owner chain.
In particular, for $0\le s\le t\le T,$ we have $f_s\prec f_t;$
in other words, $g_{T-t}\prec g_{T-s}$ holds.
Since $T$ is arbitrary, we have obtained the  subordination.
Finally, we observe that $f_t'(0)=g_{T-t}'(0)$ is absolutely continuous in $0\le t\le T$
and, in view of \eqref{eq:a1}, that
$$
\frac1{g_T'(0)}=\frac1{f_0'(0)}=\frac{f_T'(0)}{f_0'(0)}=\exp\left(\int_0^T q(0,t)dt\right).
$$
Hence,
$$
g_T'(0)
=\exp\left(-\int_0^Tq(0,t)dt\right)
=\exp\left(-\int_0^Tp(0,T-t)dt\right)
=\exp\left(-\int_0^Tp(0,t)dt\right)
$$
which tends to $0$ as $T\to+\infty$ since $p(z,t)$ is of divergence type.
Thus the assertion has been proved.
\end{proof}

As a corollary of the proof, we also have the following result,
which may be of independent interest.

\begin{corollary}\label{cor_Le4}
Under the assumptions of Lemma \ref{lem:IPDE}, we suppose, in addition, that the inequality
\begin{equation}\label{eq:arg}
\big|\arg p(z,t)\big|<\frac{\pi\alpha}{2},\quad z\in\D, \ae t\ge0,
\end{equation}
holds for a constant $0<\alpha<1.$ Then the conformal mapping $g_t$ on $\D$
extends to a $k$-quasiconformal mapping of $\C$ for each $t\ge0,$ where $k=\sin(\pi\alpha/2).$
\end{corollary}

Here and hereafter, for a constant $0\le k<1,$
a mapping $f:\C\to\C$ is called $k$-quasiconformal if
$f$ is a homeomorphism in the Sobolev class $W^{1,2}_{\mathrm loc}(\C)$ and if it satisfies
$|\partial_{\bar z}f|\le k|\partial_zf|$ almost everywhere on~$\C.$

\begin{proof}
For an arbitrary $T>0,$ we consider the family $f_t(z)=f(z,t)$ as in the above proof.
Then $|\arg q(z,t)|<\pi\alpha/2$ as well.
Now Betker's theorem (see Application 2 in \cite[p.~110]{Bet92})
implies that $f_0=g_T$ extends to a $k$-quasiconformal automorphism of $\C.$
\end{proof}

Let $f\in\N.$
That is to say, $f(z)$ is an analytic function on $\D$ such that
$f(0)=0$ and $\Re \big[f(z)/z\big]>0.$
Recall that the nonlinear resolvent $\J_r$ is defined for $f$ by \eqref{G*}.
Consider a function $p$ defined by
\begin{equation}\label{eq:p(w,r)}
p(w,r)=\frac1r\left(1-\frac{\J_r(w)}{w}\right),\quad w\in\D, r>0.
\end{equation}
We observe that the inequality $\Re p(w,r)>0$ holds because $|\J_r(w)/w|<1$ by the
Schwarz lemma.
We may set
$$
p(w,0)=\lim_{r\to0^+}\frac1r\cdot\frac{w-\J_r(w)}{w}
=\frac{f(w)}{w},
$$
so that the family $p(w,r)$ is continuous in $0\le r<\infty.$
By using Lemma \ref{lem:IPDE}, we can show the following assertion, which also
implies the inclusion relation \eqref{seq} in Theorem \ref{thm:NW}.

\begin{proposition}
The family $\J_r(w)=\J(w,r), r\ge0,$ is an inverse L\"owner chain with
the Herglotz function $p(w,r)$ of divergence type given in \eqref{eq:p(w,r)}.
In particular, $\J_r(\D)\subset\J_s(\D)$ for $0\le s\le r.$
\end{proposition}

\begin{proof}
By \eqref{G*}, we have
$$
f\big(\J(w,r)\big)=\frac{w-\J(w,r)}{r}=wp(w,r).
$$
Differentiating \eqref{G*} with respect to $r,$ we obtain
\begin{equation*}
\big[1+rf'\big(\J(w,r)\big)\big]\dot\J(w,r)+f(\J(w,r))=0.
\end{equation*}
Combining  this with \eqref{eq:J'}, we have
$$
\dot\J(w,r)=-\frac{f\big(\J(w,r)\big)}{1+rf'\big(\J(w,r)\big)}
=-w\J'(w,r)p(w,r).
$$
Since $p(0,r)=f'(0)\J_r'(0)=f'(0)/(1+rf'(0)),$ we see that
$$
\int_0^T \Re p(0,r)dr
=\Re \int_0^T\frac{f'(0)}{1+rf'(0)}dr
=\log|1+Tf'(0)|\to\infty\quad (T\to+\infty).
$$
Hence $p(w,r)$ is a Helglotz function of divergence type.
Now Lemma \ref{lem:IPDE} implies that $\J_r(w)$
forms an inverse L\"owner chain.
\end{proof}

We now state a quasiconformal extension result for the nonlinear resolvent
$\J_r(w)$ as an application of the L\"owner theory approach.

\begin{theorem}\label{thm:qc}
Suppose that $f\in\N$ satisfies the inequality
\begin{equation}\label{eq:sector2}
\left|\arg \frac{f(z)}z\right|<\frac{\pi\alpha}2,\quad z\in\D,
\end{equation}
for some constant $0<\alpha<1.$
Then the nonlinear resolvent $\J_r:\D\to\D$ for $f$ extends to a
$k$-quasiconformal mapping of $\C$ for every $r\ge 0,$
where $k=\sin(\pi\alpha/2).$
\end{theorem}

\begin{remark}\label{extension}
The condition \eqref{eq:sector2} is known to be equivalent to that the
semigroup $\big\{F_{t}\big\}_{t\geq 0}$ in $\hol(\D )$
generated by $f(z)$ can be analytically extended to the sector
$\{t\in\C: \left\vert \arg t\right\vert <\pi(1-\alpha)/2\}$ in the parameter $t$
(see \cite{E-S-Ta}).
\end{remark}

By virtue of  Corollary \ref{cor_Le4}, it is enough to show the following lemma.

\begin{lemma}
Under the assumptions of Theorem \ref{thm:qc}, the inequality
$$
\left|\arg p(w,r)\right|<\frac{\pi\alpha}2,\quad w\in\D,~ 0\le r<+\infty,
$$
holds, where $p(w,r)$ is given in \eqref{eq:p(w,r)}.
\end{lemma}

\begin{proof}
Put $q(z)=f(z)/z.$
Since the relation $z=\J_r(w)$ for $z,w\in\D$ is equivalent to the equation
$z+rf(z)=w,$ we observe that
$$
rp(w,r)=1-\frac zw=1-\frac1{1+rq(z)}=\frac{rq(z)}{1+rq(z)}\,.
$$
It is easy to verify that the sector $S_\alpha=\{\zeta: |\arg\zeta|<\pi\alpha/2\}$
is mapped univalently onto the lens-shaped domain $W_\alpha$ by the function $\zeta/(1+\zeta),$
where $W_\alpha$ is the intersection of the two disks described by
$$
\left\{\omega: \left|2\omega-\left(1\pm i\cot\frac{\pi\alpha}2\right)\right|
<\frac1{\sin(\pi\alpha/2)}\right\}.
$$
Since the boundary circles of the two disks are symmetric with respect to
the real axis and intersect at the points $0$ and $1$ with angle $\pi\alpha,$
the domain $W_\alpha$ is contained in the sector $S_\alpha.$
We now conclude that $rp(w,r)$ is contained in the sector $S_\alpha$
and hence so is $p(w,r)$ as required, because $rq(z)\in S_\alpha$ by assumption.
\end{proof}

We now combine Theorem~\ref{thm:qc} with \eqref{eq:sector} in Lemma \ref{lem:ss} to obtain the following result.

\begin{corollary}
Suppose that a holomorphic function $f:\D\to\D$ with $f(0)=0, f'(0)>0$ is starlike of
order $\alpha$ with $\frac12<\alpha<1.$
Then its nonlinear resolvent $\J_r:\D\to\D$ extends to a
$k$-quasiconformal mapping of $\C$ for every $r\ge 0,$
where $k=\sin(\pi\alpha).$
\end{corollary}

\section{Boundary regular fixed points of resolvents}

For a holomorphic function $g$ on $\D $ and a point $\zeta \in \partial \D $ we write just $g( \zeta) $ for the angular limit $\angle\lim\limits_{z\to\zeta}g(z)$ of $g$ at the point $\zeta $ and $g'(\zeta)$ for its angular derivative $\angle\lim\limits_{z\to\zeta}\frac{g(z)-g(\zeta)}{z-\zeta},$ if they exist.
We recall that $\zeta\in\partial\D$ is a {\it boundary regular null point}
(respectively, {\it boundary regular fixed point})
of a function $g\in\hol(\D,\C)$ if $g(\zeta)=0$ (resp.~ $g(\zeta)=\zeta$) and
if the finite angular derivative $g'(\zeta)$ exists.
This definition also agrees with the fact that if $F$ is a holomorphic self-mapping of $\D $,
then the function $f(z)=z-F(z)$ is a generator on $\D $ (see~\cite{SD}).

In general, regarding continuous semigroups the following fact holds (see \cite{E-S-book}).

\begin{lemma}\label{lem:new1}
Let $\{F_{t}\}_{t\geq 0}$ be the semigroup generated by an $f\in \mathcal{G}(\D)$.
Then $f$ has a boundary regular null point at $\eta\in\partial\D$ if and only if $\eta$ is a boundary regular fixed point of every semigroup element $F_t,\ t\ge0,$ with
\begin{equation*}
\left( F_{t}\right)' (\eta)=\exp \left\{ -tf'\left(\eta \right) \right\}
\end{equation*}%
Furthermore, this point is the Denjoy-Wolff point of the semigroup  $\{F_t\}_{t\geq 0}$ if and only if  $f'( \eta) \geq 0$.
\end{lemma}

In the latter case this point is also \textit{the Denjoy-Wolff point of}  $\mathcal{J}_r $ for every  $r>0$.  So, if  $f'(\eta) \geq 0$, then this point is a boundary regular fixed point for each element of both families $\{F_{t}\}_{t\geq 0}$ and $\{ {\mathcal{J}}_{r}\} _{r\geq 0}$.  However, if $f'(\zeta)<0,$ the situation is completely different (and in a sense even surprised). In this section we study the behavior of the elements of the resolvent family at a boundary
regular null point of $f\in {\mathcal{N}}$.

\begin{theorem}
\label{fixed} Let $\mathcal{J}_{r}$ be the resolvent
for a function $f\in\N$, and let $\zeta \in \partial \D $. Then $\zeta $ is a boundary
regular fixed point of $\mathcal{J}_{r}$ if and only if it is a boundary
regular null point of $f$ and $r<1/|f^{\prime }(\zeta )|<+\infty$.
Moreover, in this case, $f'(\zeta)$ is a negative real number and
$$
\J_r'(\zeta)=\frac{1}{1+rf^{\prime}(\zeta )}.
$$
\end{theorem}

\begin{proof}
Note first that the resolvent of
the rotation $R_\theta f(z)=e^{-i\theta}f(e^{i\theta}z)$ of $f(z)$ by angle $\theta$
is given as the rotation $R_\theta \J_r$ of the resolvent $\J_r$ of $f$ by angle $\theta.$
Note also that $(R_\theta f)'(z)=f'(e^{i\theta}z).$
Therefore, without loss of generality, one can assume that $\zeta =1$ by a suitable rotation
if necessary.
Since $\Re[f(z)/z]>0,$ we can express $f$ in the form
$$
f(z)=z\frac{1-F(z)}{1+F(z)},\quad z\in\D,
$$
for some $F\in\hol(\D).$

Suppose now that $\zeta=1$ is a boundary regular null point of $f(z)$ and that $r<1/\alpha,$ where $\alpha=|f'(1)|.$
Since the origin is the Denjoy-Wolff point of the generated semigroup, by Lemma~\ref{lem:new1}, $f'(1)$ is a negative real number so that $f(1)=0$ and
$\alpha=-f'(1).$

We also obtain $F(1)=1$ and $F^{\prime }(1)=2\alpha $.
By the Julia-Carath\'eodory theorem we further see that
\begin{equation}
\operatorname{Re}\frac{1+F(z)}{1-F(z)}\geq
\frac{1}{2\alpha }\operatorname{Re}\frac{1+z}{1-z}%
\quad \mbox{for all}\quad z\in \D .  \label{Julia}
\end{equation}

Using this notation, we can rewrite the resolvent equation in the form
\begin{equation}
\frac{1-z}{1-w}\left[ 1-r\frac{z(1-F(z))}{(1-z)(1+F(z))}\right] =1,
\label{resol}
\end{equation}%
where $z=\mathcal{J}_{r}(w)$. First we show that $w=1$ is a boundary fixed
point of $\mathcal{J}_{r}$. Let $\{w_{n}\}$ be a sequence in $\D$
converging to $1$. Denote $z_{n}:=\mathcal{J}_{r}(w_{n})$. Taking a
subsequence (if needed), we can assume that $z_{n}\rightarrow z_{0}$
for some point $z_0$ with $|z_{0}|\leq 1$.
Suppose (in the contrary) that $z_{0}\not=1$. Substituting $%
w=w_{n}$ into \eqref{resol}, we see that
$$
\frac{1+F(z_{n})}{1-F(z_{n})}\rightarrow \frac{rz_0}{1-z_{0}}.
$$
Therefore, letting $z=z_n\to z_0$ in \eqref{Julia}, we obtain
\begin{equation}\label{eq:z0}
(1-2r\alpha )|z_{0}|^{2}+2r\alpha \operatorname{Re}z_{0}\geq 1.
\end{equation}%
Note here that $\Re z_0\le |z_0|$ and that $|z_0|^2\le |z_0|$ because $|z_0|\le 1.$
When $1-2r\alpha\ge0,$ the inequality \eqref{eq:z0} implies that
$1\le (1-2r\alpha)|z_0|+2r\alpha|z_0|=|z_0|.$
Since $|z_0|\le 1,$ equality must hold and therefore $\Re z_0=|z_0|=1,$
which implies $z_0=1;$ a contradiction.
When $1-2r\alpha<0,$ \eqref{eq:z0} is equivalent to the condition that
$z_0$ is contained in the closed disk
$$
\left\vert z-\frac{r\alpha }{2r\alpha -1}\right\vert \leq
\dfrac{1-r\alpha}{2r\alpha -1}.
$$
It is easy to observe that
this disk intersects the closed unit disk $|z|\le 1$
only at the point $z=1.$
We have the contradiction $z_0=1$ at this time, too.
Hence we now conclude that $z_n=\J_r(w_n)$ tends to $1,$
which means that $\J_r(w)$ has the (unrestricted) limit $1$ as $w\to 1.$
This implies, \textit{inter alia}, that $1\in \partial \mathcal{J}_{r}(\D )$.
We now observe that $\mathcal{J}_{r}$ is strictly starlike; in other words,
the boundary of  the image $\J_r(\D)$ does not contain any line segment contained in a ray
emanating from the origin, because it is starlike of order $1/2.$
In particular, the segment $[0,1)$ is contained in $\mathcal{J}_{r}(\D )$,
which will be used in the sequel.

Now we prove that the boundary fixed point $%
w=1$ of $\J_r$ is regular. To this end, it is enough to show that the function
$$
g(w)=\frac{\mathcal{J}_{r}(w)-1}{w-1}
$$
has a non-vanishing finite limit (the angular derivative of $\mathcal{J}_r$)
as $w$ approaches $1$ non-tangentially.
We first observe that $g(z)$ is bounded away from $-1,$ namely,
$1/(g(z)+1)$ is bounded, in any non-tangential approach region
of the form $|\arg (1-w)|\leq c$ for some $c<\frac{\pi }{2}$. Indeed, since $%
\mathcal{J}_{r}$ is a self-mapping of the disk, $|\arg (1-\mathcal{J}%
_{r}(w))|<\frac{\pi }{2}.$ Consequently, $|\arg g(w)|<\frac{\pi }{2}+c<\pi ,$
so the claim follows.
This enables us to apply the Lindel\"of theorem to the function $1/(g(w)+1)$
to guarantee that, if it has a limit, $say A,$ along a curve $\gamma$ ending at $1$
in $\D$ then it has a non-tangential limit $A$ at $w=1$
(see \cite[Theorem 1.6]{E-S-book}).
To this end, we consider the curve $\gamma :=(\mathcal{J}_{r})^{-1}([0,1))$ in $%
\D $ ending at $1$ (see \cite[Proposition 2.14]{Pom}).
By \eqref{resol}, we get the expression
\begin{equation*}
\frac1{g(w)}=1-r\frac{z(1-F(z))}{(1-z)(1+F(z))}, \quad z=\J_r(w).
\end{equation*}%
If $w\rightarrow 1$ along the curve $\gamma $, then $\big(1-F(z)\big)/\big(1-z\big)
\to F'(1)=2\alpha $ and hence $g(w)\to 1/(1-r\alpha)>1$.
we see that the angular limit of $g(w)$ at $w=1$
is $1/(1-r\alpha).$
Thus we have shown that $\zeta =1$ is a boundary regular fixed point of $\mathcal{J}_{r}$.

Conversely, assume that $\zeta =1$ is a boundary regular fixed point of the
function $z=\mathcal{J}_{r}(w).$
Consider the auxiliary function $g$ defined by $g(w)=w-\mathcal{J}_{r}(w)$.
Then the Schwarz lemma and the Berkson-Porta formula  in Theorem C
imply that $g\in\mathcal{N}$.
In addition, $g(1)=0$ and $g'(1)=1-\mathcal{J}_{r}'(1)$.
Therefore, by  Lemma~\ref{lem:new1}, $g'(1)<0$, that is,
$\beta :=\mathcal{J}_{r}^{\prime }(1)>1$.
Then the resolvent equation \eqref{resol} implies
that if $w$ tends to $1$ along any non-tangential path, then
\begin{equation*}
\frac{1-F(\mathcal{J}_{r}(w))}{1-\mathcal{J}_{r}(w)}\rightarrow \frac{2}{r}%
\left( 1-\frac{1}{\beta }\right).
\end{equation*}%
Using the Lindel\"of theorem again, we observe that $F(w)\to 1$ and
the angular derivative of $F$ at $1$ is $2(1-1/\beta)/r.$
This implies that $w=1$ is a boundary regular
null point of $f$ \ with $f^{\prime }(1)=(1-1/\beta)/r.$
Hence $r|f^{\prime }(1)|=-rf'(1)<1$. This completes the proof.
\end{proof}

\begin{remark}
We saw in the proof that if $\alpha =-f^{\prime }(1)>0$ and $r\alpha =1$
then $w=1$ is a boundary fixed point of $\mathcal{J}_{r}$ but it is not
regular.
\end{remark}

\begin{example}\label{examp2}
 Consider now the semigroup generator $f(z)= z(1-z)$.
It has the boundary regular null point at $z=1$ with $f'(1)=-1$.
Solving the equation
$$
  z+rf(z)=z\big(1+r(1-z)\big)=w
$$
in $z,$ we find its resolvent
$$
\mathcal{J}_r(w)=\frac{r+1-\sqrt{(r+1)^2-4rw}}{2r}\,.
$$
Consider $\angle\lim\limits_{w\to1}\mathcal{J}_r(w)=\frac{r+1-\sqrt{(r-1)^2}}{2r}\,.$ Clearly, if $r\le1$, then this limit equals $1$; and otherwise, it equals $\frac1r$. So, $w=1$ is the boundary fixed point  of $\mathcal{J}_r$ if and only if $r\le1$. Moreover, if $r<1$, then
  \[
  \angle\lim_{w\to1}\frac{\mathcal{J}_r(w)-1}{w-1}=\angle\lim_{w\to1}\frac{1-r-\sqrt{(r+1)^2-4rw}}{2r(w-1)}=\frac1{1-r}\,.
  \]
  If $r=1$, then
  \[
  \angle\lim_{w\to1}\frac{\mathcal{J}_1(w)-1}{w-1}=\angle\lim_{w\to1}\frac{1}{\sqrt{1-w}}=\infty\,.
  \]
Three typical situations that occur in this example  are demonstrated in Fig.~\ref{fig1}.

\begin{figure}\centering
    \includegraphics[angle=0,width=4.8cm,totalheight=4.8cm]{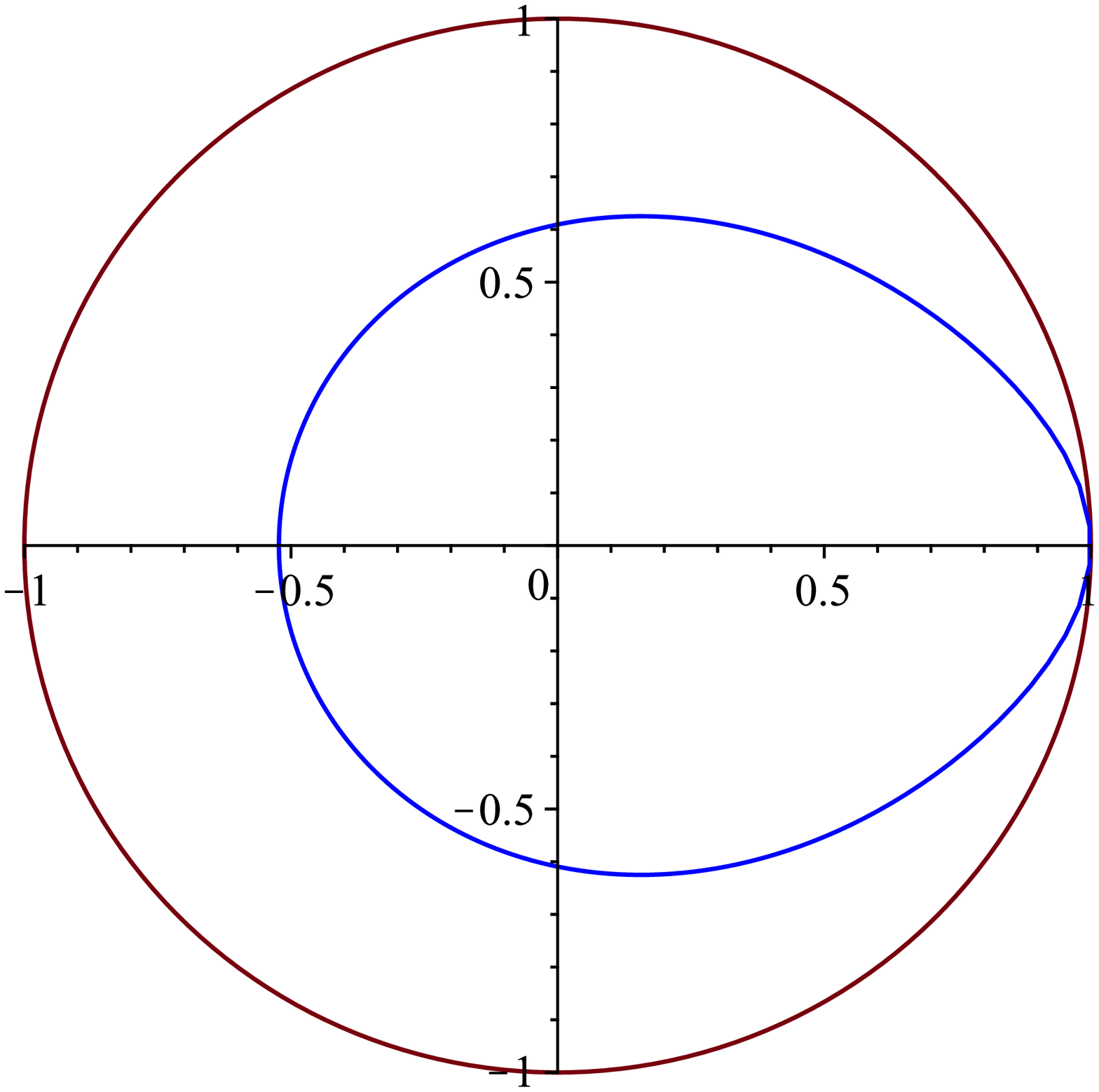}
    \includegraphics[angle=0,width=4.8cm,totalheight=4.8cm]{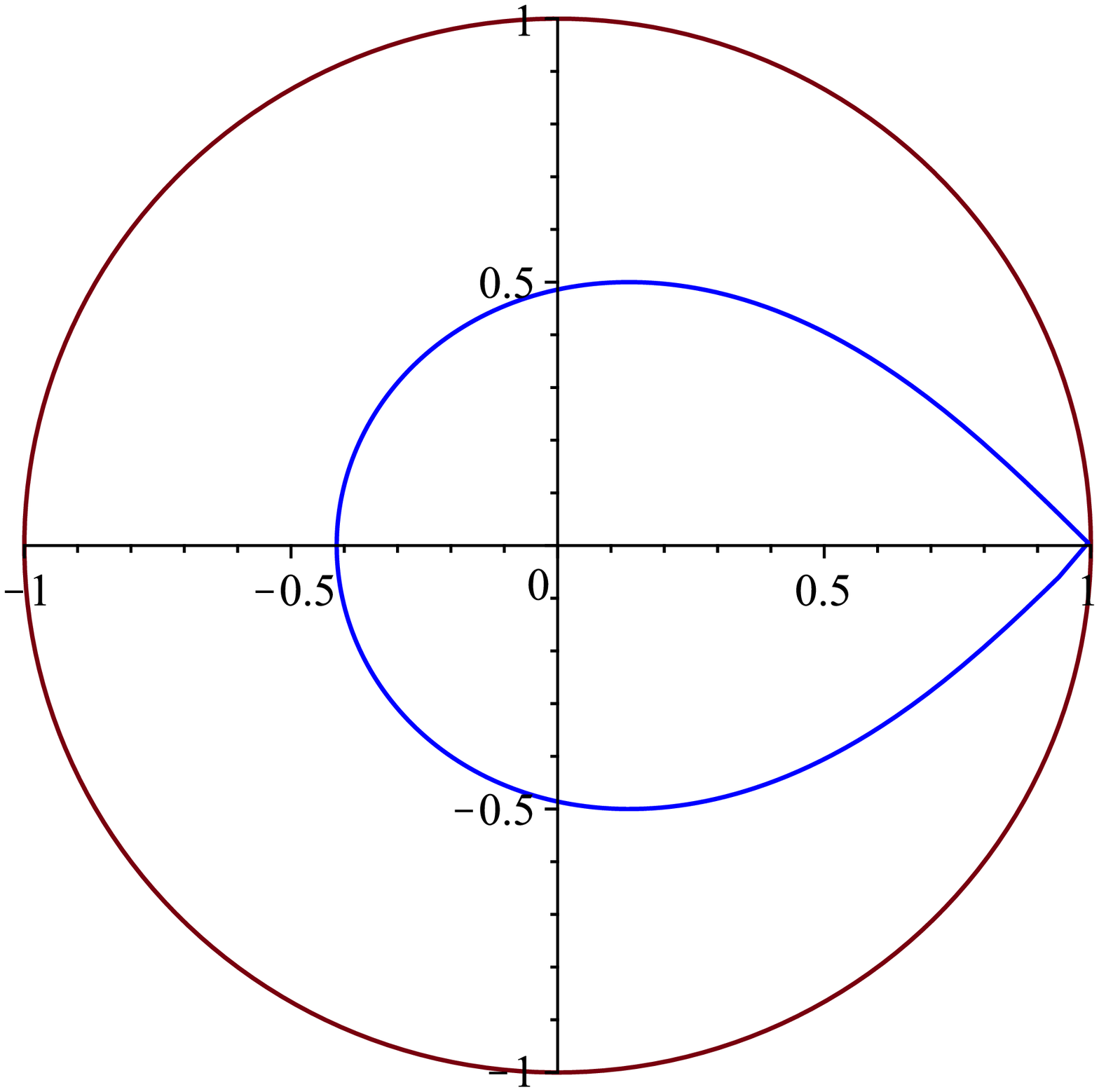}
    \includegraphics[angle=0,width=4.8cm,totalheight=4.8cm]{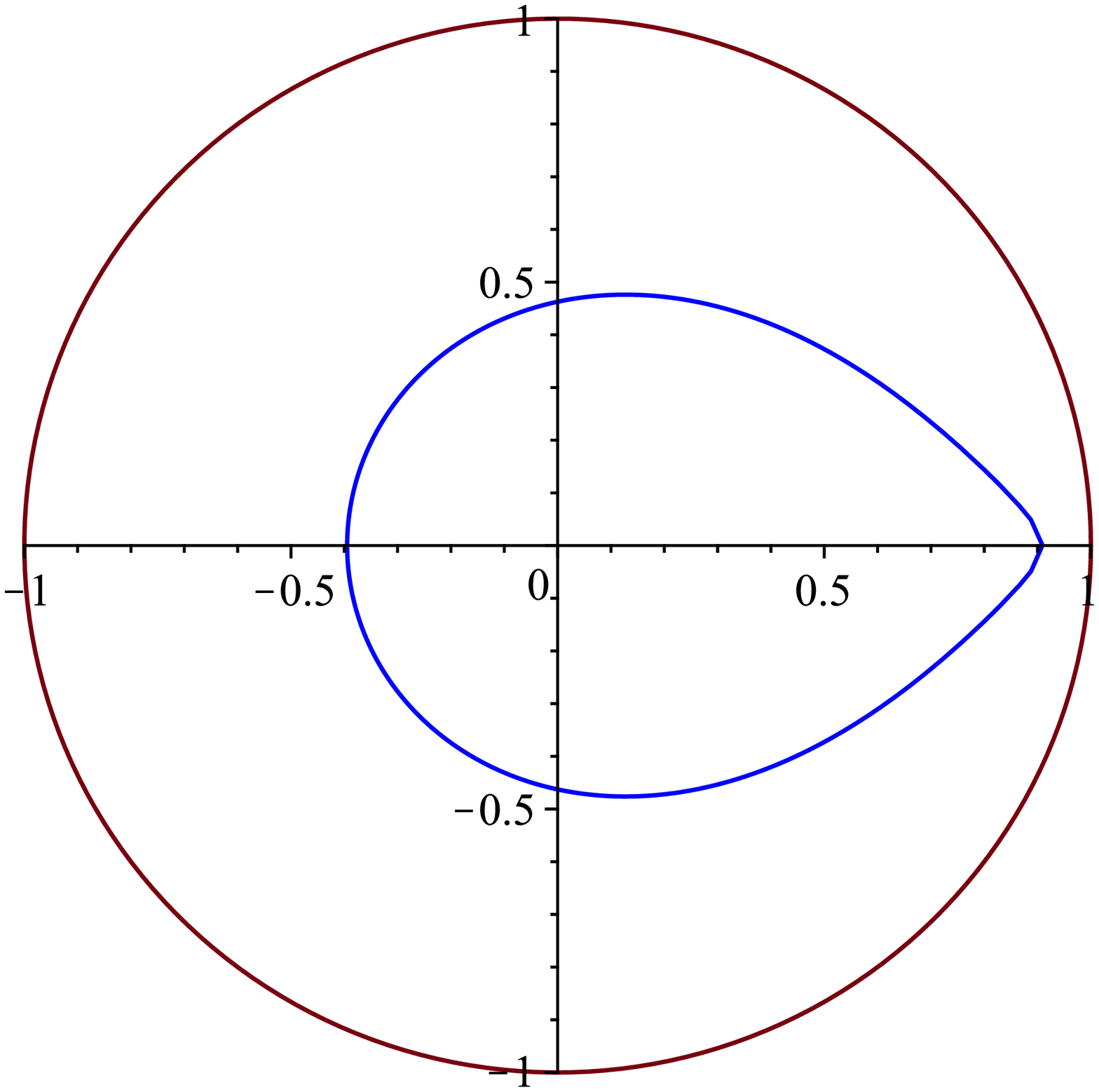}

    \caption{The images $\J_r(\D)$ for
$r=0.6<1,$ for $r=1$ and for $r=1.1>1$, respectively.}\label{fig1}
\end{figure}
\end{example}

To proceed we quote partially the result proved in \cite{E-S-Z} (see also \cite{E-S-book}).

\begin{lemma}
\label{bfid}A function $f\in {\mathcal{N}}$ has a boundary regular null point $\zeta \in \partial \D $ if and only if there is a simply connected domain $\Omega \subset \D $ such that $f$ generates a one-parameter group $S=\{F_{t}\}_{-\infty <t<\infty }$ of hyperbolic automorphisms on $\Omega $ such that the points $z=0$ and $z=\zeta $ belong to $\partial \Omega $ and are boundary regular fixed points of $S$ on $\partial \Omega $. Moreover, $f'(\zeta ) $ is a real negative number.
\end{lemma}

It follows from Lemma~\ref{lem:new1} that $F_{t}'(0)=e^{-tf'( 0) }<1$ and $F_{t}'(\zeta )=e^{-tf^{\prime }\left( \zeta \right) }>1.$

We call such a domain \textit{backward flow invariant domain} (or shortly BFID). Note that in general a BFID $\Omega$ is not unique for a point $\zeta \in \partial\D$, but there is a unique BFID  $\Omega$ (called the maximal BFID) with the above properties such that $\Omega $ has a corner of opening $\pi $ at the point $\zeta $ (see \cite{Pom}).
Other characterizations of backward flow invariant domains can be found in \cite{E-S-Z, E-R-S-Y, E-S-book}.

An interesting phenomenon occurs when we consider the resolvent family only on BFID. Namely,

\begin{proposition}\label{propo2}
Let $f\in {\mathcal{N}}$ have a boundary regular null point $\zeta\in\partial\D$ and $\Omega$ is a BFID in $\D $ corresponding to $\zeta $. If \ $\Omega $ is convex, then the restriction of the resolvent family $\J_r$ on~$\Omega$ can be continuously extended in the parameter $r\in(-\infty,0)$ such that $\zeta$ is a boundary fixed point of $\mathcal{J}_r$ for every $r<0$. Moreover, $\lim\limits_{r\to -\infty }\mathcal{J}_r(w) =\zeta $  whenever $w\in \Omega .$
\end{proposition}
\begin{proof}
Set $g=-f$ and $s=-r >0$. Then equation \eqref{G*}  becomes
\begin{equation*}
z+sg\left( z\right) =w
\end{equation*}

Since $\Omega $ is convex, it follows from Theorem \ref{teor GA} that for each $w\in \Omega$ and each $s\geq 0$ the latter equation has a unique solution $z=z_{s}\left( w\right) \in \Omega $,
which can be considered as an extension of $\mathcal{J}_{r}\left( w\right) $ for $r=-s\leq 0$.

Furthermore, since $g'(\zeta)> 0$, $\zeta$ is the Denjoy-Wolff point for the family $\left\{ z_s(w)\right\} _{s\geq 0}$. Thus the first assertion follows. Using \cite[Proposition 3.3.2]{SD}, we complete our proof.
\end{proof}
In addition, comparison of Theorem~\ref{fixed} with the last proof shows that the point $w=0$ is the boundary regular fixed point of the restriction of $\mathcal{J}_r$ on $\Omega$ whenever  $r\in\left(-\frac1{f'(0)},0\right)$  with $\J_r'(0)=\frac{1}{1+rf'(0 )}.$

To illustrate Proposition~\ref{propo2} and the last fact, return now to the semigroup generator $f(z)= z(1-z)$ and its resolvent
$$
\mathcal{J}_r(w)=\frac{r+1-\sqrt{(r+1)^2-4rw}}{2r}
$$
that were considered in Example~\ref{examp2}.  It can be easily seen that $\angle \lim\limits_{w\to1 }\mathcal{J}_r(w) =1$ for every $r<0$.  Moreover, $\mathcal{J}_r'(1)=\frac1{1-r}$ for every $r\in(-\infty,1)$  which completes the calculation in the above example. In addition, $\mathcal{J}_r(0)=0$ if $r\ge-1$ and   $\mathcal{J}_r'(0)=\frac1{1+r}$ for $r>-1$.  Using results from \cite{E-S-Z}, one can find that the maximal BFID corresponding to $\zeta=1$ is $\Omega=\left\{z:\left|z-\frac12\right|<\frac12\right\}$ (see Fig.~\ref{fig2} and compare it with Fig.~\ref{fig1}).

\begin{figure}\centering
    \includegraphics[angle=0,width=4.8cm,totalheight=4.8cm]{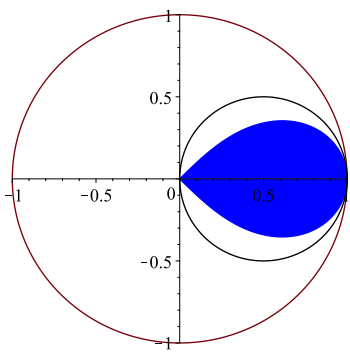}
    \caption{The maximal BFID $\Omega$ and its image $\mathcal{J}_{-1}(\Omega)$ (filled domain).}\label{fig2}
\end{figure}


\end{document}